\documentclass[10pt]{article}

\usepackage{mathptmx}
\usepackage{wrapfig}
\usepackage{eucal}
\usepackage{amsmath}
\usepackage{amscd}
\usepackage{amssymb}
\usepackage{amsthm}
\usepackage{xspace}
\usepackage[all,tips]{xy}
\usepackage[dvips]{graphicx}
\usepackage{graphics}
\usepackage{epsfig}
\usepackage{verbatim}
\usepackage{syntonly}
\usepackage{hyperref}
\usepackage{amssymb, color, url}

\newtheorem{theorem}{Theorem}
\newtheorem{proposition}[theorem]{Proposition}

\newtheorem{corollary}[theorem]{Corollary}
\newtheorem{lemma}[theorem]{Lemma}

\input xy
\xyoption{all}

\DeclareMathOperator{\PSL}{PSL}
\DeclareMathOperator{\SL}{SL}\DeclareMathOperator{\Aut}{Aut}

\DeclareMathOperator{\stab}{stab}
\DeclareMathOperator{\Mod}{Mod}
\DeclareMathOperator{\Sym}{Sym}

\DeclareMathOperator{\Out}{Out}

\DeclareMathOperator{\Z}{\mathbb{Z}}

\def\Z{\mathbb{Z}}
\def\N{\mathbb{N}}

\usepackage{fancyhdr}

\fancypagestyle{plain}

\begin{document}


\title{\textbf{Quotients of mapping class groups from $\Out(F_n)$}}
\author{Khalid Bou-Rabee\thanks{K.B. supported in part by NSF grant
    DMS-1405609}\and Christopher J. Leininger \thanks{C.L. supported in part by NSF
grant DMS-1510034.}}
\maketitle


\begin{abstract}
We give a short proof of Masbaum and Reid's result that mapping class groups involve any finite group, appealing to free quotients of surface groups and a result of Gilman, following Dunfield--Thurston.
\end{abstract}

\vskip.1in
{\small{\bf keywords:}
\emph{mapping class groups, involve, finite groups}}
\vskip.1in

Let $\Sigma_g$ be a closed oriented surface of genus $g$ and $F_n$ a nonabelian free group of rank $n$.
The fundamental group, $\pi_1(\Sigma_g)$, is residually free \cite{MR0140562} and $F_n$ has a wealth finite index subgroups \cite[pp. 116]{MR2109550}.
In \cite{MR2257389}, N.~Dunfield and W.~Thurston consider the action of the mapping class group $\Mod(\Sigma_g)$ on the set of finite index normal subgroups of $\pi_1(\Sigma_g)$ with finite simple quotients, and in particular those containing the kernel of an epimorphism $\pi_1(\Sigma_g) \to F_g$.  Their observations relating to work of R.~Gilman \cite{MR0435226}, give rise to finite index subgroups of $\Mod(\Sigma_g)$ that surjects symmetric groups of arbitrarily large order; see the discussion following the proof of \cite[Theorem~7.4]{MR0140562}.

\begin{theorem}[Dunfield--Thurston] \label{thm:main1} 
For all $g \geq 3$, $r \geq 1$, there exists an epimorphism $\phi : \pi_1(\Sigma_g) \to F_g$ and a prime $q$, so that
$$
\left\{ N \triangleleft \pi_1(\Sigma_g) \mid \ker \phi < N \mbox{ and } \pi_1(\Sigma_g)/N \cong \PSL(2,q) \right\}
$$
has at least $r$ elements, and its (finite index) stabilizer in $\Mod(\Sigma_g)$ acts as the full symmetric group on this set.
\end{theorem}

We explain the proof of this in Section~\ref{S:handlebody free}.  In this note, we observe that since every finite group embeds in some finite symmetric group, Theorem~\ref{thm:main1} provides a new elementary proof of a result of G.~Masbaum and A.~Reid \cite{MR2967055}. 
Recall that a group $G$ \emph{involves} a group $H$ if there exists a finite index subgroup $L \leq G$ and a surjective map $\phi: L \to H$.

\begin{corollary} [Masbaum--Reid] \label{cor:involve}
Let $\Sigma_{g,m}$ be a surface of genus $g$ with $m$ punctures.  If $3g-3+m \geq 1$ (or $g=1$ and $m=0$) then $\Mod(\Sigma_{g,m})$ involves any finite group.
\end{corollary}
\noindent The few mapping class groups not covered by the corollary are finite groups; see, e.g.~\cite{MR2850125}. Corollary~\ref{cor:involve} is also proved using arithmetic methods by F.~Grunewald, M.~Larsen, A.~Lubotzky, and J.~Malestein \cite{MR3426060}.

Further applications of the quotients from Theorem~\ref{thm:main1} include new proofs of residual finiteness and separability of handlebody groups; see \S \ref{sec:handlebodysubgroups} for theorem statements and proofs.

\smallskip

\noindent
{\bf Acknowledgements.} The authors would like to thank Alan Reid and Alex Lubotzky for helpful conversations and Benson Farb for suggesting improvements on an earlier draft.

\section{Preliminaries}

\subsection{$G$-defining subgroups}

Here we collect some results surrounding definitions and discussions in R.~Gilman \cite{MR0435226}.
Let $G$ and $F$ be groups.
A \emph{$G$-defining subgroup of $F$} is a normal subgroup $N$ of $F$ such that $F/N$ is isomorphic to $G$.  Let $X(F,G)$ denote the set of all $G$--defining subgroups of $F$.
The automorphism group $\Aut(F)$ acts on normal subgroups of $F$ while preserving their quotients, and hence on the set $X(F,G)$ of $G$-defining subgroups of $F$.  Since inner automorphisms act trivially, the action descends to an action of the outer automorphism group of $F$, $\Out(F)$, on $X(F,G)$.
If $G$ is finite, and $F$ is finitely generated, one obtains a finite permutation representation of $\Out(F)$.
Let $F_n$ be the free group of rank $n$.  The following is Theorem 1 of \cite{MR0435226}.

\begin{theorem}[Gilman] \label{GilmanTheorem}
For any $n \geq 3$ and prime $p \geq 5$, $\Out(F_n)$ acts on the $\PSL(2,p)$-defining subgroups of $F_n$ as the alternating or symmetric group, and both cases occur for infinitely many primes.
\end{theorem}

From the proof, Gilman obtains the following strengthened form of residual finiteness for $\Out(F_n)$.

\begin{corollary} [Gilman] \label{cor:Out(Fn) RF sym-alt} For any $n \geq 3$, the group $\Out(F_n)$ is residually finite alternating and residually finite symmetric via the quotients from Theorem~\ref{GilmanTheorem}.
\end{corollary}
This means that for any $\phi \in \Out(F_n) - \{1\}$, there exist primes $p$ so that the action of $\Out(F_n)$ on $X(F_n,\PSL(2,p))$ is alternating (and also primes $p$ so that the action is symmetric), and $\phi$ acts nontrivially.

We will also need the following well-known fact, obtained from the classical embedding of a free group into $\PSL(2, \Z)$ as a subgroup of finite index, (c.f.~A.~Peluso \cite{MR0199245}).

\begin{lemma} \label{L:F_n is RF} For any $n \geq 2$, any element $\alpha \in F_n - \{1\}$, and all but finitely many primes $p$, there exists a $\PSL(2,p)$--defining subgroup of $F_n$ not containing $\alpha$.
\end{lemma}

\begin{proof}
Let $F_n$ be a finite index, free subgroup of rank $n$ in the free group $F_2 := \left< a, b \right>$.
Identify $F_n$ with its image in $\PSL(2, \Z)$ under the injective homomorphism $F_2 \to \PSL(2,\Z)$ given by
$$
a \mapsto \begin{pmatrix} 1 & 2 \\ 0 & 1 \end{pmatrix}, \text{ and } \;\;
b \mapsto \begin{pmatrix} 1 & 0 \\ 2 & 1 \end{pmatrix}.
$$ 
Let $\alpha \in F_n - \{1 \}$ be given and let $A \in \SL(2,\Z)$ be a matrix representing $\alpha$.  Since $\alpha \neq 1$, we may assume that either $A$ has a nonzero off-diagonal entry $d \neq 0$, or else a diagonal entry $d > 1$.
Then for any prime $p$ not dividing $d$ in the former case, or $d \pm 1$ in the latter, we have that $\pi_p(\alpha)$ is nontrivial in the quotient $\pi_p : \PSL(2, \Z) \to \PSL(2,p)$; that is, $\alpha \notin \ker \pi_p$.

Since $F_n$ has finite index in $F_2$, there exists $m \geq 1$ so that the matrices
$$
\begin{pmatrix} 1 & m \\ 0 & 1 \end{pmatrix} \text{ and } \begin{pmatrix} 1 & 0 \\ m & 1 \end{pmatrix}
$$
represent elements of $F_n$ in $\PSL(2,\Z)$.   For any prime $p$ not dividing $m$, the $\pi_p$--image of these elements generate $\PSL(2,p)$.  Thus, for all but finitely many primes $p$, $\ker \pi_p \cap F_n$ is a $\PSL(2,p)$--defining subgroup not containing $\alpha$.
\end{proof}

\subsection{Handlebody subgroups and maps to free groups}  \label{S:handlebody free}

Let $\Sigma = \Sigma_g$ be a closed surface of genus $g \geq 2$ and let $H = H_g$ be a handlebody of genus $g$.  Given a homeomorphism $\phi \colon \Sigma \to \partial H \subset H$, the induced homomorphism is a surjection $\phi_* \colon \pi_1(\Sigma) \to \pi_1(H) \cong F_g$.   As is well-known, every epimorphism $\pi_1(\Sigma) \to F_g$ arises in this way (see e.g.~Lemma 2.2 in \cite{MR1885215}).  The kernel, $\Delta_\phi = \ker(\phi_*)$ is an $F_g$--defining subgroup, and is the subgroup generated by the simple closed curves on $\Sigma$ whose $\phi$--images bound disks in $H$.  We write $H_\phi$ for the handlebody $H$, equipped with the homeomorphism $\phi \colon \Sigma \to \partial H$.

Let $\Mod(H_\phi)$ denote the subgroup of the mapping class group $\Mod(\Sigma)$ consisting of the isotopy classes of homeomorphisms that extend over $H_\phi$ (via the identification $\phi \colon \Sigma \to \partial H$).   Equivalently, $\Mod(H_\phi)$ consists of those mapping classes $[f]$ such that $f_*(\Delta_\phi) = \Delta_\phi$; that is $\Mod(H_\phi)$ is the stabilizer in $\Mod(\Sigma)$ of $\Delta_\phi$.  Any element $[f] \in \Mod(H_\phi)$ induces an automorphism we denote $\Phi_*([f]) \in \Out(F_g)$, which defines a homomorphism $\Phi_* \colon \Mod(H_\phi) \to \Out(F_g)$.  The main result of \cite{MR0159313} implies the next proposition.

\begin{proposition} \label{prop:surj}
For any $g \geq 0$, and homeomorphism $\phi \colon \Sigma \to \partial H$,  $\Phi_* : \Mod(H_\phi) \to \Out(F_g)$ is surjective.
\end{proposition}

The kernel of $\Phi_*$, the set of mapping classes in $\Mod(H_\phi)$ that act trivially on $\pi_1(H)$ is also a well-studied subgroup denoted $\Mod_0(H_\phi)$. 

Recall that $X(F_g,G)$ and $X(\pi_1(\Sigma),G)$ are the sets of $G$--defining subgroups of $F_g$ and $\pi_1(\Sigma)$, respectively.   Define
\[ X^\phi(\pi_1(\Sigma),G) :=  \{ \phi_*^{-1}(N) \mid N \in X(F_g,G)\} \subset X(\pi_1(\Sigma),G), \]
Alternatively, this is precisely the set of $G$--defining subgroups containing $\Delta_\phi$:
\[ X^\phi(\pi_1(\Sigma),G) = \{ N \in X(\pi_1(\Sigma),G) \mid \Delta_\phi < N \}.\]

\begin{lemma} \label{L:Mod(H) stabilizer} The handlebody subgroup is contained in the stabilizer
\[ \Mod(H_\phi) < \stab X^\phi(\pi_1(\Sigma),G).\]
Moreover, if $\Out(F_g)$ acts on $X(F_g,G)$ as the full symmetric group, then $\Mod(H_\phi)$ (and hence $\stab X^\phi(\pi_1(\Sigma),G)$) acts on $X^\phi(\pi_1(\Sigma),G)$ as the full symmetric group.
\end{lemma}
\begin{proof}
Let $N \in X^\phi(\pi_1(\Sigma),G)$ and let $[f] \in \Mod(H_\phi)$, where $f$ is a representative homeomorphism.  Since $f_*(\Delta_\phi) = \Delta_\phi$, we have $\Delta_\phi < f_*(N)$, and $f_*(N) \in X^\phi(\pi_1(\Sigma),G)$.  Thus, $f_*$ preserves $X^\phi(\pi_1(\Sigma),G)$, as required.

The last statement follows immediately from Proposition~\ref{prop:surj} and the fact that the bijection from the correspondence theorem $X^\phi(\pi_1(\Sigma),G) \to X(F_g,G)$ is $\Phi_*$--equivariant.
\end{proof}
\section{Mapping class groups involve any finite group: The proofs of Theorem~\ref{thm:main1} and Corollary~\ref{cor:involve}.}

Here we give the proof of Theorem~\ref{thm:main1}, following Dunfield--Thurston (see \cite[pp. 505-506]{MR2257389}).

\begin{proof}[Proof of Theorem~\ref{thm:main1}.]
Fix $g \geq 3$ and let $\Pi$ be the infinitely many primes for which $\Out(F_g)$ acts on the $\PSL(2,p)$-defining subgroups as the symmetric group, guaranteed by Theorem~\ref{GilmanTheorem}.
As a consequence of Corollary \ref{cor:Out(Fn) RF sym-alt}, the cardinality of $X(F_g, \PSL(2,p))$ is unbounded over any infinite set of primes $p$, and hence there exists a prime $p \in \Pi$ where the number of  $\PSL(2,p)$-defining subgroups is $R \geq r$.

By Lemma \ref{L:Mod(H) stabilizer}, $\stab X^\phi(\pi_1(\Sigma),\PSL(2,p))$ acts on $X^\phi(\pi_1(\Sigma),\PSL(2,p))$ as the symmetric group, defining a surjective homomorphism
\[ \stab X^\phi(\pi_1(\Sigma),\PSL(2,p)) \to \Sym(X^\phi(\pi_1(\Sigma),\PSL(2,p))) \cong S_R.\]
Since $X(\pi_1(\Sigma),\PSL(2,p))$ is a finite set, $\stab X^\phi(\pi_1(\Sigma),\PSL(2,p)) < \Mod(\Sigma)$ has finite index, completing the proof.
\end{proof}

\begin{proof}[Proof of Corollary~\ref{cor:involve}]
For $g,m$ as in the statement and any $r \in \N$, we show that there is a finite index subgroup of $\Mod(\Sigma_{g,m})$ that surjects a symmetric group on at least $r$ elements.  Since any finite subgroup is isomorphic to a subgroup of some such symmetric group, this will prove the theorem.

First observe that for any $m,g \geq 0$, the kernel of the action of $\Mod(\Sigma_{g,m})$ on the $m$ punctures of $\Sigma_{g,m}$ is a finite index subgroup $\Mod'(\Sigma_{g,m}) < \Mod(\Sigma_{g,m})$.  Furthermore, if $0 \leq m < m'$, there is a surjective homomorphism $\Mod'(\Sigma_{g,m'}) \to \Mod'(\Sigma_{g,m})$ obtained by ``filling in'' $m'-m$ of the punctures; see \cite{MR2850125}.

Now, because $\Mod'(\Sigma_{0,4}) \cong F_2$, and the symmetric group on $r$ elements is 2--generated, it follows that $\Mod'(\Sigma_{0,4})$ surjects $S_r$.  From the previous paragraph, it follows that $\Mod'(\Sigma_{0,m})$ surjects $S_r$ for all $m \geq 4$.  Similarly, $\Mod(\Sigma_{1,0}) \cong \SL(2,\Z)$, which has a finite index subgroup isomorphic to $F_2$, and so there is a finite index subgroup of $\Mod(\Sigma_{1,m})$ that surjects $S_r$ for all $m \geq 0$.
As shown in \cite{MR0292082}, there is a surjective homomorphism $\Mod(\Sigma_{2,0}) \to \Mod(\Sigma_{0,6})$, and consequently, we may find surjective homomorphisms from finite index subgroups of $\Mod(\Sigma_{2,m})$ to $S_r$ for all $m \geq 0$.  From this and the previous paragraph, it suffices to assume $g \geq 3$ and $m  = 0$.  The required surjective homomorphism to a symmetric group in this case follows from Theorem~\ref{thm:main1}, completing the proof.
\end{proof}

\section{Separating handlebody subgroups and residual finiteness}
\label{sec:handlebodysubgroups}

The finite quotients of $\Mod(\Sigma)$ coming from surjective homomorphisms $\pi_1(\Sigma_g) \to F_g$ also allow us to deduce the following result of \cite{MR2271769}.  Recall that a subgroup $K < F$ is said to be separable in $F$ if for any $\alpha \in F - K$, there exists a finite index subgroup $G < F$ containing $K$ and not containing $\alpha$.
\begin{theorem} [Leininger-McReynolds] \label{thm:sep handlebody subgroups} For any $g \geq 2$ and homeomorphism to the boundary of a handlebody, $\phi \colon \Sigma \to \partial H$, the groups $\Mod(H_\phi)$ and $\Mod_0(H_\phi)$ are separable in $\Mod(\Sigma_g)$.
\end{theorem}
\noindent While the proof of separability of $\Mod(H_\phi)$ in $\Mod(\Sigma_g)$ works for all $g \geq 2$, separability of $\Mod_0(\Sigma_g)$ only follows from the discussion here when $g \geq 3$.
\begin{proof} Suppose $\Sigma = \Sigma_g$ for $g \geq 2$, and observe that for any $p$, any $[h] \in \stab X^\phi(\pi_1(\Sigma),\PSL(2,p))$, and any $\alpha \in \Delta_\phi$, we have $h_*(\alpha) \in K$ for all $K \in X^\phi(\pi_1(\Sigma),\PSL(2,p))$.  By Lemma~\ref{L:Mod(H) stabilizer} this is true for all $[h] \in \Mod(H_\phi)$.  

Now let $[f] \in \Mod(\Sigma) - \Mod(H_\phi)$, so that $f_*(\Delta_\phi) \not < \Delta_\phi$.  Let $\gamma \in \Delta_\phi$ be an element such that (the conjugacy class of) $f_*(\gamma)$ is not in $\Delta_\phi$.  (In fact, well-defining $f_*(\gamma)$ requires a choice of basepoint preserving representative homeomorphisms for the mapping class of $f$, which we make arbitrarily.)   Then $\phi_*(f_*(\gamma)) \in F_g - \{1\}$, and so by Lemma~\ref{L:F_n is RF}, we can find a prime $p$ and a $\PSL(2,p)$--defining subgroup $N \in X(F_g,\PSL(2,p))$ so that $\phi_*(f_*(\gamma)) \not \in N$.  Therefore,
\[ f_*(\gamma) \not \in \phi_*^{-1}(N) \in X^\phi(\pi_1(\Sigma),\PSL(2,p)),\] and hence $[f] \not \in \stab X^\phi(\pi_1(\Sigma),\PSL(2,p))$.  Since $\stab X^\phi(\pi_1(\Sigma),\PSL(2,p))$ is a finite index subgroup containing $\Mod(H_\phi)$ (by Lemma~\ref{L:Mod(H) stabilizer}) and not containing $[f]$, and since $[f]$ was arbitrary, it follows that $\Mod(H_\phi)$ is separable.

Since $\Mod_0(H_\phi) < \Mod(H_\phi)$ and since $\Mod(H_\phi)$ is separable, it suffices to consider an element $[f] \in \Mod(H_\phi) \setminus \Mod_0(H_\phi)$, and produce a finite index subgroup of $\Mod(\Sigma)$ containing $\Mod_0(H_\phi)$ and not containing $[f]$.  For all $p$, $\Mod_0(H_\phi)$ is contained in the subgroup of $\stab X^\phi(\pi_1(\Sigma),\PSL(2,p))$ consisting of those mapping classes that act trivially on $X^\phi(\pi_1(\Sigma),\PSL(2,p))$.  Since $[f] \not \in \Mod_0(H_\phi)$, $\Phi_*([f]) \neq 1$ in $\Out(F_g)$.  For $g \geq 3$, Corollary~\ref{cor:Out(Fn) RF sym-alt} implies that for some $p$, $\Phi_*([f])$ acts nontrivially on $X(F_g,\PSL(2,p))$.  Therefore, $[f]$ acts nontrivially on $X^\phi (\pi_1(\Sigma),\PSL(2,p))$, and so the finite index subgroup $G < \Mod(\Sigma)$ consisting of those mapping classes preserving the subset $X^\phi(\pi_1(\Sigma),\PSL(2,p))$ and acting trivially on this does not contain $[f]$, proving that $\Mod_0(H_\phi)$ is separable.  
\end{proof}

Mapping class groups were shown to be residually finite by Grossman as a consequence of the fact that surface groups are conjugacy separable; see \cite{MR0405423}.  
\begin{theorem} [Grossman] \label{thm:res finite} Mapping class groups are residually finite.
\end{theorem}
Residual finiteness of $\Mod(\Sigma_g)$ follows immediately from separability of the handlebody subgroups $\Mod(H_\phi)$, and the following.
\begin{lemma} \label{lem:trivial handlebody intersection}
The intersection of all handlebody subgroups $\Mod(H_\phi)$, over all homeomorphisms $\phi \colon \Sigma_g \to \partial H$ is trivial if $g \geq 3$, and isomorphic to $\Z/2\Z$ if $g = 2$.  The intersection of handlebody subgroups $\Mod_0(H_\phi)$ is trivial for all $g \geq 2$. 
\end{lemma}
\begin{proof} In \cite{MR837978}, Masur proved that the limit set of the handlebody subgroup in the Thurston boundary of Teichm\"uller space is a nowhere dense subset.  The intersection of all handlebody subgroups is a normal subgroup and so is either finite, or else has limit set equal to the entire Thurston boundary.  By Masur's result, we must be in the former case, and hence the intersection of handlebody subgroups is finite.  But $\Mod(\Sigma_g)$ has no nontrivial finite, normal subgroups if $g \geq 3$, while for $g =2$, the only nontrivial, finite normal subgroup is the order-two subgroup generated by the hyperelliptic involution.   This proves the first statement.  The second follows from the first and the fact that the hyperelliptic involution of $\Sigma_2$ induces a nontrivial automorphism of $F_2 \cong \pi_1(H)$, for any homeomorphism $\phi \colon \Sigma_2 \to H$.
\end{proof}

\begin{proof}[Proof of Theorem~\ref{thm:res finite} for $\Mod(\Sigma_g)$, with $g \geq 2$.]  An equivalent formulation of residual finiteness is that the intersection of all finite index subgroups is trivial.  Therefore Theorem~\ref{thm:sep handlebody subgroups} and Lemma~\ref{lem:trivial handlebody intersection} immediately implies the result.
\end{proof}

\bibliography{refs.bib}
\bibliographystyle{amsalpha}


\noindent
Khalid Bou-Rabee \\
Department of Mathematics, CCNY CUNY \\
E-mail: khalid.math@gmail.com \\

\noindent
Christopher Leininger \\
Department of Mathematics, University of Illinois Urbana-Champaigne \\
E-mail: c.j.leininger95@gmail.com \\

\end{document}